\documentclass[12pt]{amsart}
\usepackage{tikz}
\usepackage{caption}
\usepackage{subcaption}
\usepackage[colorlinks=true, pdfstartview=FitV, linkcolor=blue, citecolor=black, plainpages=false, pdfpagelabels=true, urlcolor=blue]{hyperref}

\DeclareMathOperator{\sech}{sech}

\usepackage{graphicx}
\usepackage{color}
\usepackage{amsmath}
\usepackage{amssymb}
\newcommand{\beq}{\begin{eqnarray*}}
\newcommand{\feq}{\end{eqnarray*}}
\newcommand{\beqn}{\begin{eqnarray}}
\newcommand{\feqn}{\end{eqnarray}}

\newtheorem{theorem}{Theorem}[section]
\newtheorem{lemma}[theorem]{Lemma}

\theoremstyle{definition}

\newtheorem{example}[theorem]{Example}

\theoremstyle{remark}

\numberwithin{equation}{section}

\allowdisplaybreaks

\begin{document}

\title[Riccati dynamics of  Euler-Poisson equations]{On the Riccati dynamics of 2D Euler-Poisson equations with attractive forcing}
\author{Yongki Lee}
\address{Department of Mathematical Sciences, Georgia Southern University, Statesboro,  30458}
\email{yongkilee@georgiasouthern.edu}
\begin{abstract} 
 
 The Euler-Poisson (EP) system describes the dynamic behavior of many important physical flows. In this work, a Riccati system that governs two-dimensional EP equations is studied.
  The evolution of divergence is governed by the Riccati type equation with several nonlinear/nonlocal terms. Among these, the vorticity accelerates divergence while others further amplify the blow-up behavior of a flow. The growth of these blow-up amplifying terms are related to the Riesz transform of density, which lacks a uniform bound makes it difficult to study global solutions of the multi-dimensional EP system.  We show that the Riccati system can afford to have global solutions, as long as the growth rate of blow-up amplifying terms is not higher than exponential, and admits global smooth solutions for a large set of initial configurations.  To show this, we construct an auxiliary system in 3D space and find an invariant space of the system, then comparison with the original 2D system is performed. Some numerical examples are also presented.

\end{abstract}
\maketitle


\section{Introduction}

We are concerned with the threshold phenomenon in two-dimensional Euler-Poisson (EP) equations. The pressureless Euler-Poisson equations in multi-dimensions are
\begin{subequations} \label{101}
    \begin{equation} \label{101a}
        \rho_t + \nabla \cdot (\rho \mathbf{u})=0,
    \end{equation}
    \begin{equation} \label{101b}
        \mathbf{u}_t + \mathbf{u} \cdot \nabla \mathbf{u} = k \nabla \Delta^{-1} (\rho-c_b),
    \end{equation}
\end{subequations}
which are the usual statements of the conservation of mass and Newton's second law. Here $k$ is a physical constant which parameterizes the repulsive $k>0$ or attractive $k<0$ forcing, governed by the Poisson potential $\Delta^{-1}(\rho -c_b)$ with constant $c_b>0$ which denotes background state.
  The local density $\rho=\rho(t,\mathbf{x})$ : $\mathbb{R}^+ \times \mathbb{R}^2 \mapsto \mathbb{R^+} $  and the velocity field $\mathbf{u}(t,\mathbf{x})$ : $\mathbb{R}^+ \times \mathbb{R}^2 \mapsto \mathbb{R}^{2}$  are the unknowns. This hyperbolic system \eqref{101} with non-local forcing describes the dynamic behavior of many important physical flows, including  plasma with collision, cosmological waves, charge transport,  and the collapse of stars due to self gravitation.

There is a considerable amount of literature available on the solution behavior of Euler-Poisson system. Global existence due to damping relaxation and with non-zero back-ground can be found in \cite{Wang01}. For the model without damping relaxation, construction of a global smooth irrotational solution in three dimensional space can be found in \cite{Guo98}. Some related results on two dimensional case can be found in \cite{JLZ14, LW14, IP13}. One the other hand, we refer to \cite{E96, WW06, MP90} for singularity formation and nonexistence results.

We focus our attention on the questions of global regularity versus finite-time blow-up  of solutions for \eqref{101}. Many of the results mentioned above leave open the question of global regularity of solutions to \eqref{101} subject to more general conditions on initial configurations, which are not necessarily confined to a ``sufficiently small" ball of any preferred norm of initial data. In this regard, we are concerned here with so called Critical Threshold (CT) notion, originated and developed in a series of papers by Engelberg, Liu and Tadmor  \cite{ELT01, LT02, LT03} and more recently in various models \cite{TT14, BL20, YT19}. The critical threshold in \cite{ELT01} describes the conditional stability of the one-dimensional Euler-Poisson system, where the answer to the question of global vs. local existence depends on whether the initial data crosses a critical threshold. Following \cite{ELT01}, critical thresholds have been identified for several \emph{one-dimensional} models, e.g., $2 \times 2$ quasi-linear hyperbolic relaxation systems \cite{LL09}, Euler equations with non-local interaction and alignment forces \cite{TT14, CCTT16}, traffic flow models \cite{YT19}, and damped Euler-Poisson systems \cite{BL20}.

Moving to the \emph{multi-dimensional setup}, the main difficulty lies with the non-local nature of the forcing term $\nabla \Delta^{-1} \rho$, and this feature was the main motivation for studying the so called ``restricted" or ``modified" EP models \cite{Y16, Y17, LT03, Tan14}, where the nonlocal forcing term is replaced by a local or semi-local one. The regularity of the (original) Euler-Poisson equations in $n>1$ dimensions remains an outstanding open problem.

The goal of this paper is showing that, under a suitable assumption, two-dimensional Euler-Poisson system with attractive forcing can \emph{afford} to have global smooth solutions for a large set of initial configurations. In section \ref{section 2}, we seek the evolution of $\nabla \mathbf{u}$ and derive a closed ordinary differential equations (ODE) system which is nonlinear and nonlocal, and relate/review many previous works with the derived ODE system. In section \ref{section 3}, we discuss the motivation and highlights of the present work. In addition to this, we state our main result about global solutions to the EP system. The details of the proof of the main result are carried out in Sections \ref{section 4}. Some numerical examples are provided in Section \ref{section 5}, that illustrate the behavior of the solutions, $\Delta^{-1}(\rho - c_b)$ and $\nabla\Delta^{-1}(\rho - c_b)$.

$$$$

\section{Problem formulation and related works}\label{section 2}
In this work, we consider two-dimensional Euler-Poisson equations with attractive forcing \eqref{101}.
We are mainly concerned with a Riccati system that governs  $\nabla \mathbf{u}$.
In order to trace the evolution of $\mathcal{M}:=\nabla \mathbf{u}$, we differentiate \eqref{101b}, obtaining
\begin{equation}\label{Meqn_pre}
\partial_t \mathcal{M} + \mathbf{u} \cdot \nabla \mathcal{M} + \mathcal{M}^2 = k \nabla \otimes \nabla \Delta^{-1} (\rho -c_b) = k R[\rho-c_b],
\end{equation}
where $R[\cdot] $ is the $2 \times 2$ Riesz matrix operator, defined as
$$R[h]:=\nabla \otimes \nabla \Delta^{-1}[h]=\mathcal{F}^{-1}\bigg{\{} \frac{\xi_i \xi_j}{|\xi|^2} \hat{h}(\xi)  \bigg{\}}_{i,j=1,2} .$$

We let $\frac{D}{Dt}[\cdot] = [\cdot]'$ be the usual material derivative, $\frac{\partial}{\partial t} + \mathbf{u} \cdot \nabla$.
We are concerned with the initial value problem \eqref{Meqn_pre} or
\begin{equation}\label{Meqn}
\frac{D}{Dt}\mathcal{M} + \left(
                   \begin{array}{cc}
                     \mathcal{M}_{11}^2 + \mathcal{M}_{12}\mathcal{M}_{21} & (\mathcal{M}_{11}+\mathcal{M}_{22})\cdot \mathcal{M}_{12} \\
                      (\mathcal{M}_{11}+\mathcal{M}_{22}) \cdot \mathcal{M}_{21} & \mathcal{M}_{12}\mathcal{M}_{21} +\mathcal{ M}_{22}^2 \\
                   \end{array}
                 \right)
 = k\left(
                         \begin{array}{cc}
                           R_{11} [\rho-c_b]& R_{12} [\rho-c_b]\\
                           R_{21}[\rho -c_b] & R_{22} [\rho-c_b]\\
                         \end{array}
                       \right).
 \end{equation}
 subject to initial data
$$(\mathcal{M}, \rho)(0, \cdot) = (M_0 , \rho_0).$$
The global nature of the Riesz matrix $R[ \cdot ]$, makes the issue of regularity for Euler-Poisson equations such an intricate question to solve.

We introduce several quantities with which we characterize the behavior of the velocity gradient tensor $\mathcal{M}$. These are the trace $d:=\mathrm{tr} \mathcal{M} = \nabla \cdot \mathbf{u}$, the vorticity $\omega : = \nabla \times \mathbf{u} = \mathcal{M}_{21} - \mathcal{M}_{12}$ and quantities $\eta:= \mathcal{M}_{11} - \mathcal{M}_{22}$ and $\xi := \mathcal{M}_{12} + \mathcal{M}_{21}$. 
Taking the trace of \eqref{Meqn}, one obtain
\begin{equation}\label{Deqn}
\begin{split}
 d'&= - (\mathcal{M}^2 _{11} + \mathcal{M}^2 _{22}) -2\mathcal{M}_{12}\mathcal{M}_{21} + k (R_{11} [\rho-c_b]  + R_{22} [\rho-c_b])\\
&=-\bigg{\{ } \frac{(\mathcal{M}_{11} + \mathcal{M}_{22})^2}{2} +  \frac{(\mathcal{M}_{11} - \mathcal{M}_{22})^2}{2}   \bigg{\}} +  \frac{(\mathcal{M}_{21} - \mathcal{M}_{12})^2}{2} -  \frac{(\mathcal{M}_{12} + \mathcal{M}_{21})^2}{2}   + k(\rho-c_b)\\
&=-\frac{1}{2}d^2 -\frac{1}{2}\eta^2 + \frac{1}{2}\omega^2 -\frac{1}{2}\xi^2 + k(\rho-c_b).
\end{split}
\end{equation}

We can see that the equation \eqref{Deqn} is a Ricatti-type equation.
One can view the evolution of $d$ as the result of a contest between negative and positive terms in \eqref{Deqn}.  Indeed,  the vorticity accelerates divergence while $\eta$ and $\xi$ suppress divergence and enhance the finite time blow-up of a flow. The growth of $\eta$ and $\xi$ are related to the Riesz transform of density and non-locality of these terms make it difficult to study global solutions of the multidimensional EP system.

Our approach in this paper is to study the evolutions of $d=\nabla \cdot \mathbf{u}$ and it shall be carried out by tracing the dynamics of  $\eta$, $\omega$ and $\xi$. From matrix equation \eqref{Meqn}, and \eqref{101a}, we obtain
\begin{subequations}
    \begin{equation}\label{Eeqn}
  	 \eta' +\eta d =k(R_{11}[\rho-c_b] - R_{22}[\rho-c_b]),
    \end{equation}
    \begin{equation}\label{Veqn}
    	\omega' + \omega d = k(R_{21}[\rho-c_b]-R_{12}[\rho-c_b])=0,
    \end{equation}
    \begin{equation}\label{Xeqn}
     \xi' + \xi d = k (R_{12}[\rho-c_b]+R_{21}[\rho-c_b]),
    \end{equation}
     \begin{equation}\label{Reqn}
    \rho' + \rho d=0.
        \end{equation}
        \end{subequations}
Here, one can explicitly  calculate $R[\cdot]$, (see \cite{Y16} for the detailed calculations) i.e.,
\begin{equation}
(R_{i j}  [h])(\mathbf{x})  = p.v.\int_{\mathbb{R}^2} \frac{\partial^2}{\partial y_j \partial y_i}G(\mathbf{y})h(\mathbf{x} - \mathbf{y}) \, d \mathbf{y} +  \frac{h(\mathbf{x})}{2\pi}\int_{|\mathbf{z}| =1}   z_i z_j  \, d\mathbf{z},
\end{equation}
where $G(\mathbf{y})$ is the Poisson kernel in two-dimensions, and is given by
$$G(\mathbf{y}) = \frac{1}{2\pi} \log |\mathbf{y}|.$$
Due to the singular nature of the integral, we are lack of $L^{\infty}$ estimate of the $R_{ij}[\cdot]$.
        
From \eqref{Veqn} and \eqref{Reqn}, we derive
\begin{equation}\label{ome_rho}
\frac{\omega}{\omega_0}=\frac{\rho}{\rho_0}.
\end{equation}
One can also rewrite $\eta$ and $\xi$ in terms of $\rho$, by explicitly solving \eqref{Eeqn} and \eqref{Xeqn} (see \cite{Y16}), we obtain
\begin{equation}\label{ex_rho}
\eta(t) =  \bigg{(} \frac{\eta_0}{\rho_0}  + \int^t _0 \frac{f_1(\tau)}{\rho(\tau)} \, d\tau \bigg{)}\rho(t) \ and \ \xi(t) =   \bigg{(} \frac{\xi_0}{\rho_0}  + \int^t _0 \frac{f_2(\tau)}{\rho(\tau)} \, d\tau \bigg{)}\rho(t),
\end{equation}
where
\begin{equation}\label{f_1}
f_1(t):=k(R_{11}[\rho-c_b] - R_{22}[\rho-c_b]) = \frac{k}{\pi} p.v.\int_{\mathbb{R}^2} \frac{-y^2 _1 + y^2 _2}{(y^2 _1 + y^2 _2)^2} \rho(t, \mathbf{x}(t) - \mathbf{y} ) \, d\mathbf{y},
\end{equation}
and
\begin{equation}\label{f_2}
f_2(t):=k (R_{12}[\rho-c_b]+R_{21}[\rho-c_b]) = \frac{k}{\pi} p.v.\int_{\mathbb{R}^2} \frac{-2y_1 y_2}{(y^2 _1 + y^2 _2)^2} \rho(t, \mathbf{x}(t) - \mathbf{y} ) \, d\mathbf{y}.
\end{equation}
Here, all functions of consideration are evaluated along the characteristic, that is, for example,
$f_i (t)=f_i(t, \mathbf{x}(t))$ and $\eta(t)=\eta(t, \mathbf{x}(t))$, etc.

Using \eqref{ome_rho} and \eqref{ex_rho} we can rewrite \eqref{Deqn} in a manner that all non-localities are absorbed in the coefficient of $\rho^2$. That is, together with \eqref{Reqn}, we obtain closed system
\begin{equation}\label{ode1_intro}
\left\{
  \begin{array}{ll}
    d' = -\frac{1}{2}d^2 + A(t)\rho^2 + k(\rho -c_b), \\
    \rho' = -\rho d,\\
  \end{array}
\right.
\end{equation}
where
\begin{equation}\label{A_eqn}
A(t):=\frac{1}{2} \bigg{[} \bigg{(} \frac{\omega_0}{\rho_0} \bigg{)}^2 - \bigg{(} \frac{n_0}{\rho_0} + \int^t _0 \frac{f_1(\tau)}{\rho(\tau)} \, d \tau \bigg{)}^2 - \bigg{(} \frac{\xi_0}{\rho_0} + \int^t _0 \frac{f_2(\tau)}{\rho(\tau)} \, d \tau \bigg{)}^2  \bigg{]}.
\end{equation}

In this work, we are concern with \eqref{ode1_intro}, subject to initial data
$$(\nabla \mathbf{u}, \rho)(0,\cdot)=(\nabla \mathbf{u}_0, \rho_0).$$

To put our study in a proper perspective we recall several recent works in the form of \eqref{ode1_intro}. It turns out that many of so-called restricted/modified models can be reinterpreted within the scope of \eqref{ode1_intro}.

$\bullet$ Chae and Tadmor \cite{CT08}  proved the finite time blow-up for solutions of $k<0$ case, assuming \emph{vanishing initial vorticity}. Indeed, setting $\omega_0 =0$ in \eqref{A_eqn}  gives $A(t) \leq 0$, and this allows to derive 
$$d' \leq -\frac{1}{2}d^2 +  k(\rho -c_b).$$
Using this ordinary differential inequality, upper-threshold for finite time blow-up of solution was identified. Later Cheng and Tadmor \cite{CT09} improved the result of \cite{CT08} using the delicate ODE phase plane argument.

$\bullet$  Liu and Tadmor \cite{LT02, LT03} introduced the restricted Euler-Poisson (REP) system which is obtained from \eqref{Meqn} by restricting attention to the local isotropic trace $\frac{k}{2} (\rho -c_b) I_{2 \times 2}$ of the global coupling term $kR[\rho- c_b]$. One can also obtain the REP by letting $f_i \equiv 0$, $i=1,2$ in \eqref{A_eqn}. That is,
$$d' = -\frac{1}{2}d^2 + \frac{\beta}{2}\rho^2 + k(\rho -c_b), \ \ \beta=\frac{\omega^2 _0 - \eta^2 _0 - \xi^2 _0}{\rho^2 _0}.$$
Thus, $A(t) \equiv  \frac{\beta}{2}$ in REP. The dynamics of $(\rho, d)$ of this ``localized" EP system was studied, and it was shown that in the repulsive case, the REP system admits so called critical threshold phenomena.

$\bullet$ Slight generalization of the REP was introduced in \cite{Y17}. This ``weakly" restricted EP can also be obtained by letting $f_1\equiv 0$ \emph{only} in \eqref{A_eqn}. Indeed, $f_1\equiv 0$ implies $A(t) \leq \frac{1}{2}[(\frac{\omega_0}{\rho_0})^2 - (\frac{\eta_0}{\rho_0})^2]$ and
$$d' \leq -\frac{1}{2}d^2 + \frac{\alpha}{2}\rho^2 + k(\rho -c_b), \ \ \alpha=\frac{\omega^2 _0 - \eta^2 _0 }{\rho^2 _0}.$$
Threshold conditions for finite time blow-up were identified for attractive and repulsive cases.

$\bullet$ While the dynamics of $d$ in the above reviewed models are governed by \emph{local} quantities, the model in \cite{Y16} strives to maintain some \emph{global} nature of $A(t)$. That is, the author assumed that
$$\bigg{|} \int^t _0 \frac{f_i(\tau)}{\rho(\tau)} \, d \tau \bigg{|} \leq C \int^t _0 \frac{1}{\rho(\tau)} \, d \tau, \ \ i=1,2$$
for some constant $C$, and obtained upper-thresholds for finite time blow-up for attractive and repulsive cases.

$\bullet$  Tan \cite{Tan14} assumed that
$$\bigg{|}\frac{f_i (t)}{\rho(t)}  \bigg{|} \leq C, \ \ i=1,2$$
and proved a global existence of solution for  repulsive case using some scaling argument. The above assumption is equivalent to $A(t)\equiv -C(t+D)^2$ for some constant $D$.\\

In all above mentioned restricted/modified models and the vanishing vorticity case, as well as the current work, the difficulty associated with having Riesz transforms is not present. This issue remains open. We should point out that, however, all these works have strived to advance the understanding on EP system through \emph{more general} $A(t)$ which  \emph{further amplifies} the blow-up behavior by the first term $-\frac{1}{2}d^2$ in \eqref{ode1_intro}. In this paper, we are concerned with $A(t) \sim -e^t$ case and sub-critical configurations which admit global smooth solutions.

$$$$

\section{ Main result and remarks }\label{section 3}
We first address some motivation of this work.  The difficulty lies with the nonlocal/singular nature of the Riesz transform, which fails to map $L^{\infty}$ data to $L^{\infty}$. Thus, main obstacle in handling the dynamics of $d$ in \eqref{ode1_intro} is the lack of an accurate description for the propagations of $f_i (t, \mathbf{x}(t))$ in \eqref{f_1} and \eqref{f_2}. This, in turn, makes difficult to answer the questions of global regularity versus finite-time breakdown of solutions for \eqref{ode1_intro}.

From \eqref{A_eqn}, we know the initial value, and the uniform \emph{upper bound} of $A(t)$. That is,

$$A(0)=\frac{1}{2} \bigg{[} \bigg{(} \frac{\omega_0}{\rho_0} \bigg{)}^2 - \bigg{(} \frac{n_0}{\rho_0}  \bigg{)}^2 - \bigg{(} \frac{\xi_0}{\rho_0}  \bigg{)}^2  \bigg{]}$$
and
\begin{equation}\label{A_upper}
A(t) \leq \frac{1}{2} \bigg{(} \frac{\omega_0}{\rho_0} \bigg{)}^2, \ for \ all \ t \geq 0,
\end{equation}
as long as $A(t)$ exists.

However, we do not know if there exists any lower bound of $A(t)$. It is possible that $A(t) \rightarrow -\infty$ in finite/infinite time or remains uniformly bounded below for all time. This is because, as mentioned earlier, there is no $L^\infty$ bound of $f_i$. For each fixed $t$, we know that  $f_i (t,\cdot) \in \mathrm{BMO}(\mathbb{R}^2)$ (bounded mean oscillation, see e.g. \cite{S93}), and this implies that
$$f_i (t,\cdot) \in L^p _{loc} (\mathbb{R}^2), \ \ 1\leq p < \infty.$$ 

Since $A(t)$ is bounded above, we are left with only two possible cases
 under non-vacuum condition $\rho_0 >0$ (thus $\rho(t) >0$ from the second equation of \eqref{ode1_intro}):
 
\textbf{Case I:} Finite time blow-up of $A(t)$. That is, $$\lim_{t \rightarrow T -} A(t) = -\infty,$$ where $T < \infty$. This corresponds to 
$$\lim_{t \rightarrow T -}\bigg{|} \int^t _0 f_i (\tau, \mathbf{x}(\tau)) \, d\tau \bigg{|} = \infty,$$
and this is possible because for each $t$, $f_i (t, \cdot)$ need not be locally bounded, so  $f_i(t, \mathbf{x}(t))$ can be unbounded along some part of the characteristic.  In this case, we can easily see that  $$\lim_{t \rightarrow T -} d(t) = -\infty$$ as well, and $\rho, d$ blow-up in finite time. Indeed, suppose not, then since $T <\infty$, 
$$\rho(T) = \rho_0 e^{-\int^{T} _0 d(\tau, \mathbf{x}(\tau)) \, d \tau} >\epsilon,$$
for some $\epsilon>0$. Applying this and $\lim_{t \rightarrow T -} A(t) = -\infty$ to \eqref{ode1_intro}, we obtain $\lim_{t \rightarrow T -} d' (t) = -\infty$ and this is contradiction.

\textbf{Case II:} $A(t)$ uniformly bounded below, or blows-up at infinity. That is, there exists some function  $h(t)$ such that
$A(t) \geq h(t)$, for all $t\geq 0$ and 
$$h(t) \rightarrow -\infty \  as \  t \rightarrow \infty.$$ 

The main contribution of this work is investigating \eqref{ode1_intro} under the condition in Case II, and we show that the Riccati structure  can afford to have global solutions even if $A(t) \rightarrow -\infty$, depending on $A(t)$'s rate of decreasing. More precisely, we show that the nonlinear-nonlocal system \eqref{ode1_intro} admits global smooth solutions for a large set of initial configurations provided that
\begin{equation}\label{exp_condi_1}
A(t) \geq -\alpha e^{\beta t}, \ for \ all \ t,
\end{equation}
where $\alpha$ and $\beta$ are some positive constants.



From now on, in \eqref{ode1_intro}, we assume that $k=-1$ and $c_b=1$, because these constants are not essential in our analysis. Also, we set $\alpha=\beta=1$. We shall consider 
\begin{equation}\label{ode2_intro}
\left\{
  \begin{array}{ll}
    d' = -\frac{1}{2}d^2 + A(t)\rho^2 - (\rho -1), \\
    \rho' = -\rho d,\\
  \end{array}
\right.
\end{equation}
subject to initial data $(\nabla \mathbf{u}, \rho)(0,\cdot)=(\nabla \mathbf{u}_0, \rho_0),$
where
\begin{equation}\label{A_eqn_2}
A(t):=\frac{1}{2} \bigg{[} \bigg{(} \frac{\omega_0}{\rho_0} \bigg{)}^2 - \bigg{(} \frac{n_0}{\rho_0} + \int^t _0 \frac{f_1(\tau)}{\rho(\tau)} \, d \tau \bigg{)}^2 - \bigg{(} \frac{\xi_0}{\rho_0} + \int^t _0 \frac{f_2(\tau)}{\rho(\tau)} \, d \tau \bigg{)}^2  \bigg{]},
\end{equation}
with assumption that
\begin{equation}\label{exp_condi_2}
A(t) \geq - e^{ t}, \ for \ all \ t.
\end{equation}

Here, we note that  \eqref{exp_condi_2} already assume that $A(0)\geq-1$, that is,
$$\frac{1}{2} \bigg{[} \bigg{(} \frac{\omega_0}{\rho_0} \bigg{)}^2 - \bigg{(} \frac{n_0}{\rho_0}  \bigg{)}^2 - \bigg{(} \frac{\xi_0}{\rho_0}  \bigg{)}^2  \bigg{]} \geq-1.$$
But this does not restrict our result, since one can always find $\alpha$ and $\beta$ that satisfy \eqref{exp_condi_1} for any $A(0)$.

To present our result, we write \eqref{Meqn_pre} or \eqref{Meqn} here again, with the second equation of \eqref{ode2_intro}, to establish the two-dimensional Euler-Poisson system:
\begin{equation}\label{EP_system}
\left\{
  \begin{array}{ll}
    \mathcal{M}' + \mathcal{M}^2 =-R[\rho -1], \\
    \rho' = -\rho\cdot \mathrm{tr}(\mathcal{M}),\\
  \end{array}
\right.
\end{equation}
subject to initial data $(\mathcal{M},\rho)(0,\cdot)=(\mathcal{M}_0 , \rho_0)$.
We note that the global regularity follows from the standard boot-strap argument, once an \emph{a priori} estimate on $\| \mathcal{M}(\cdot, \cdot) \|_{L^{\infty}}$ is obtained. Also, $\mathcal{M}=\nabla \mathbf{u}$ is completely controlled by $d=\nabla \cdot \mathbf{u}$ and $\rho$, see \cite{LT03}:
$$\|\mathcal{M}(t,\cdot) \|_{L^{\infty}[0,T]} \leq C_T \cdot \| (\mathrm{tr}\mathcal{M} , \rho) \|_{L^{\infty} [0,T]}.$$


\begin{figure}[ht]
\begin{center}
\includegraphics[width=160mm]{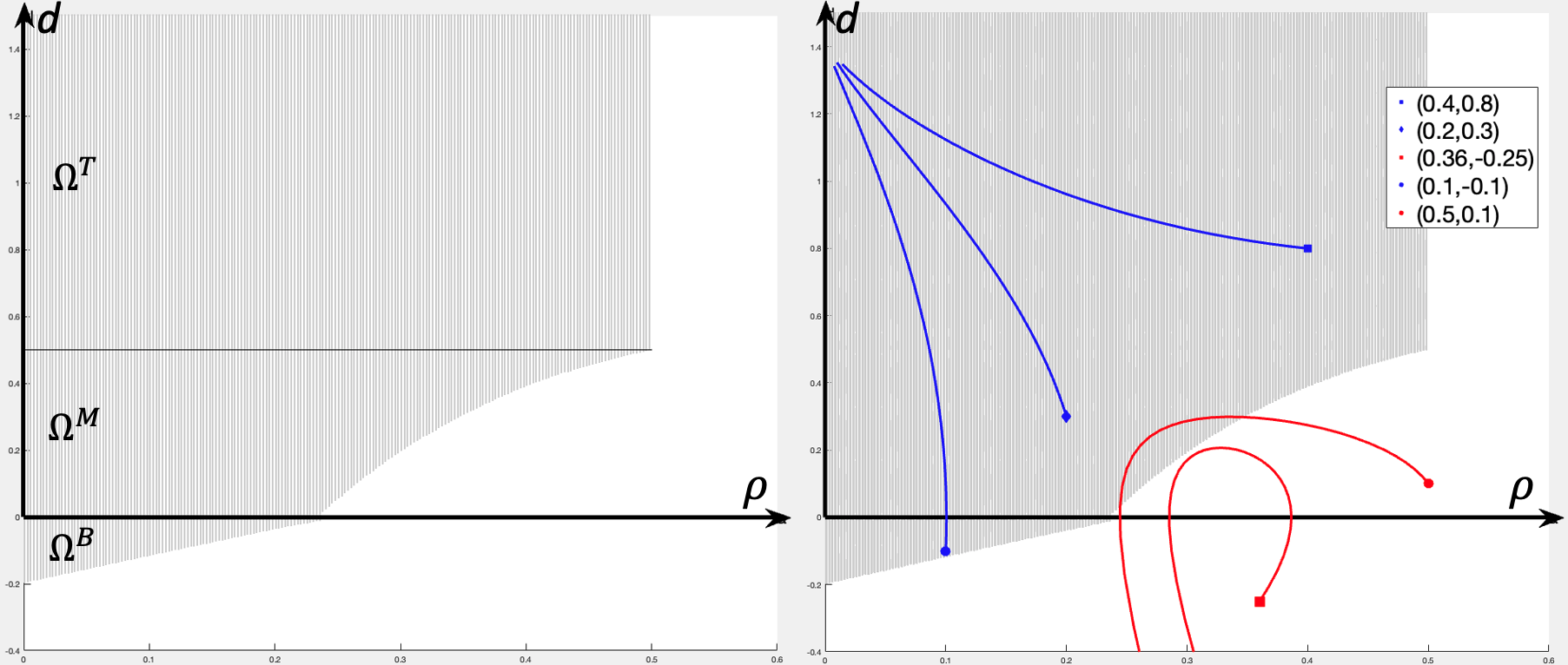}
\end{center}
\caption{Left: The sub-critical data sets $\Omega^T$, $\Omega^M$ and $\Omega^B$ in Theorem \ref{thm_main}, Right: When $A(t)\equiv -e^t$, some numerically calculated solutions of \eqref{ode2_intro} with several initial configurations are plotted. Here, the blue trajectories denote globally regular solutions, whereas the red trajectories denote blow-up solutions.}\label{ot_om_ob}
\end{figure}

Our goal of this work is to prove the following result.

\begin{theorem}\label{thm_main}
Consider the Euler-Poisson system,  \eqref{EP_system}  with \eqref{exp_condi_2}. 
If 
$$(\rho_0, d_0) \in \Omega^T \cup \Omega^M \cup \Omega^B,$$
 then the solution of the Euler-Poisson system remains smooth for all time. Here,
 $$\Omega^T : =  \{(\rho,d) \ | \ 0 <\rho <1/2, \ d \geq 1/2 \},$$
$$\Omega^{M} : = \bigg{\{}(\rho,d) \ \bigg{|} \ 0 <\rho < 1/2, \ \max\bigg{\{} 0, \frac{1}{2} - \bigg{(}\frac{3}{8}  - \frac{\rho}{2}\bigg{)}\log\bigg{[} \frac{1}{2} \big{(} \frac{1}{\rho^2}  - \frac{1}{\rho} \big{)} \bigg{]} \bigg{\}} < d \leq \frac{1}{2} \bigg{\}},$$
and
$$\Omega^{B} := \bigg{\{ } (\rho,d) \ \bigg{|} \ 0 <\rho <1/2, \ \rho -\frac{1}{2} < d <0,   \ \frac{(\frac{1}{2} -d)}{\{ \frac{3}{8} -\frac{1}{2}(\rho - d)  \}}\leq \log \bigg{[} \frac{1}{2} \bigg{(}  \frac{1}{(\rho  - d)^2} -\frac{1}{\rho- d}\bigg{)}  \bigg{]} \bigg{\}},$$
as shown in $\mathrm{Figure}$ \ref{ot_om_ob}.
 \end{theorem}

\begin{figure}[ht]\begin{center}
\includegraphics[width=160mm]{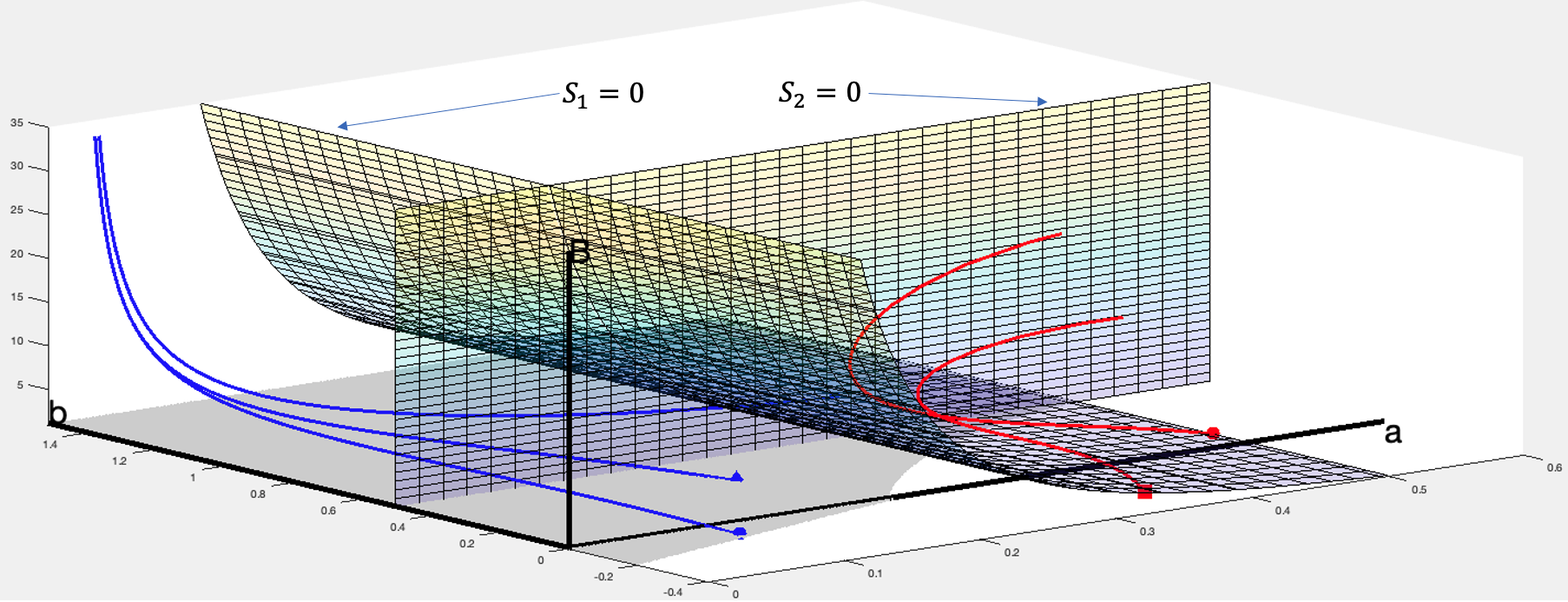}
\caption{Plotting of some solutions $(\rho(t), d(t), e^t)$ with initial data in Figure \ref{ot_om_ob} (when $A(t)\equiv -e^{t}$). Here, the gray region denotes $\Omega^T \cup \Omega^M \cup \Omega^B$ in  Figure \ref{ot_om_ob}. Surfaces $S_1 =0$ and $S_2 =0$ determine a 3D invariant space for the auxiliary system mentioned in the remark. Later, we extend the invariant space to obtain the desired result.}\label{invariant}
\end{center}
\end{figure}

\textbf{Remarks.} Some remarks are in order at this point.

1. We note that $A(t) \geq - e^t$ \emph{further amplifies} the blow-up behavior by the first term $-\frac{1}{2}d^2$ in \ref{ode2_intro}. Thus, the main contribution of the theorem is that the Riccati structure \ref{ode2_intro} affords to have global solutions even though $A(t) \rightarrow -\infty$ at infinity, as long as the the rate is not greater than exponential. In addition to this, the set of initial configurations in the theorem contains  negative divergences.

2. As discussed in Case I, an unconditional finite time blow ups of $\rho$ and $d$ can occur if $A(t) \rightarrow -\infty$ in finite time. Theorem \ref{thm_main} supports that an unconditional finite time blow up cannot be ignited by exponential time growth of $|A(t)|$.

3. In Figure \ref{ot_om_ob}, the subcritical initial data sets are plotted in $\rho$-$d$ plane. When $A(t)\equiv -e^t$, some numerically calculated solutions of \ref{ode2_intro} are plotted from various initial data. One can see that the trajectories initiated from gray region converge to $(0,\sqrt{2})$, while the red trajectories go to $(\infty, -\infty)$ in finite time. It is interesting to see that the red trajectory from $(0.5, 0.1)$ enters the gray region, but it does blow up, due to the  time dependent coefficient $A(t)$ in \eqref{ode2_intro}.

4. In order to prove the theorem, we introduce a $3 \times 3$ auxiliary system
\begin{equation*}
\left\{
  \begin{array}{ll}
    \dot{b} = -\frac{1}{2}b^2 - B a^2 + k(a -1), \\
    \dot{a} = -ba,\\
    \dot{B} =B,
  \end{array}
\right.
\end{equation*}
(with $k=-1$) and find a three-dimensional invariant space of the system, where all trajectories if they start from inside this space will stay encompassed at all time, see Figure \ref{invariant}. Then, we compare the auxiliary system with the original system \eqref{ode2_intro}. The key parts of the proof are constructing the surfaces $S_1 =0$ and $S_2 =0$ in Figure \ref{invariant} that determine the three-dimensional invariant space of the auxiliary system, extending the invariant space, and establishing monotonicity between the auxiliary system and the original system via comparison. 

5. The aforementioned comparison argument (Lemma \ref{lemma_comp}) between the auxiliary system and the original system does not work for the \emph{repulsive forcing} case, $k=1$. 





$$$$

\section{Proof of Theorem \ref{thm_main}}\label{section 4}

We start this section by considering the following nonlinear ODE system with the time dependent coefficient,
\begin{equation}\label{ode_exp}
\left\{
  \begin{array}{ll}
    \dot{b} = -\frac{1}{2}b^2 - e^t a^2 -a +1, \\
    \dot{a} = -ba.\\
  \end{array}
\right.
\end{equation}

Setting $B(t)=e^t$, one can rewrite the system as follows:
\begin{equation}\label{3by3_exp}
\left\{
  \begin{array}{ll}
    \dot{b} = -\frac{1}{2}b^2 - B a^2 -a +1, \\
    \dot{a} = -ba,\\
    \dot{B} =B,
  \end{array}
\right.
\end{equation}
with $(a , b , B)\big{|}_{t=0} = (a_0 , b_0 , B_0 =1).$

We shall find a set of initial data for which the solution of \eqref{3by3_exp} exists for all time. We first define $\Omega_{abB}$ space as
$$\Omega_{abB}:=\{(a,b,B) \ | \ a\in (0, \infty), \ b \in (-\infty, \infty) \ and \ B \in [1, \infty)  \}.$$
In addition to this, we let 
$$\vec{l}(t):=\langle \dot{a}(t), \dot{b}(t), \dot{B}(t) \rangle.$$
That is, $\vec{l}(t)$ is the tangent vector of the trajectory or solution $(a(t), b(t), B(t))$.

We now define $S_1 =0 $ and $S_2 =0$ surfaces in $\Omega_{abB}$,  as shown in Figure \ref{invariant} :
First,
$$S_1 (a,b,B):= \frac{1}{2} \bigg{(}\frac{1}{a^2}  -\frac{1}{a} \bigg{)} -B =0.$$
Since $B\geq 1$, we define $S_1 =0$ surface only when $0<a\leq \frac{1}{2}$.

Next, we let
$$S_2 (a,b,B):= b-\frac{1}{2}=0.$$
We first show that these two surfaces form an invariant space of \eqref{3by3_exp}. 

On $S_1 =0$ surface,  it holds $\nabla S_1 \cdot \vec{l}(t) >0$, provided that $b \geq \frac{1}{2}$. Indeed,
\begin{equation*}
\begin{split}
\nabla S_1 \cdot \vec{l}(t) &= \bigg{\langle} -\frac{1}{a^3} + \frac{1}{2a^2}, 0, -1 \bigg{\rangle} \cdot \langle -ba, -\frac{1}{2}b^2 - B a^2 -a +1, B \rangle\\
&=\frac{b}{a^2} - \frac{b}{2a} -B.
\end{split}
\end{equation*}
Thus, on $S_1 =0$ surface, we have
\begin{equation}\label{s1_grad}
\nabla S_1 \cdot \vec{l}(t) \bigg{|}_{S_1 =0} = \frac{b}{a^2} - \frac{b}{2a} - \frac{1}{2a^2} + \frac{1}{2a}=\frac{1}{2a^2}\big{(}2b -ab -1 +a \big{)}>0,
\end{equation}
provided that $0<a \leq \frac{1}{2}$ and $b \geq \frac{1}{2}$.

On $S_2 =0$ surface, we consider
\begin{equation*}
\begin{split}
\nabla S_2 \cdot \vec{l}(t) &= \langle 0, 1, 0 \rangle \cdot \langle -ba, -\frac{1}{2}b^2 - B a^2 -a +1, B \rangle\\
&=-\frac{1}{2}b^2 -Ba^2 -a +1.
\end{split}
\end{equation*}
Thus, if $B \leq \frac{1}{2} \big{(}\frac{1}{a^2}  -\frac{1}{a} \big{)}$ and $0 < a \leq \frac{1}{2}$, then on $S_2 =0$ surface,
 \begin{equation}\label{s2_grad}
\nabla S_2 \cdot \vec{l}(t) \bigg{|}_{S_2 =0} \geq -\frac{1}{2} \bigg{(}\frac{1}{2} \bigg{)}^2 - \frac{a^2}{2} \bigg{(} \frac{1}{a^2} -\frac{1}{a} \bigg{)} -a + 1 = \frac{3}{8} - \frac{a}{2} > 0.
\end{equation}

Hence by \eqref{s1_grad} and \eqref{s2_grad}, we determine the invariant space:
\begin{equation}\label{omega_0}
\Omega_0:=\bigg{\{} (a,b,B) \ \bigg{|} \ 0 < a \leq \frac{1}{2} , \ b\geq \frac{1}{2},  \ B \leq \frac{1}{2} \big{(}\frac{1}{a^2}  -\frac{1}{a} \big{)} \bigg{\}}.
\end{equation}

Now, the following lemma is elementary, but useful to extend the invariant space.
\begin{lemma}\label{b_inc}
If $B \leq \frac{1}{2} \big{(}\frac{1}{a^2}  -\frac{1}{a} \big{)}$  and $b^2 < -a +1$, then
$$\dot{b}(t) >0.$$
\end{lemma}
\begin{proof} Consider
\begin{equation}\label{b_dot_ineq}
\begin{split}
\dot{b}(t) = -\frac{1}{2}b^2 - B a^2 -a +1 \geq -\frac{1}{2}b^2 -\frac{a^2}{2} \bigg{(}\frac{1}{a^2} - \frac{1}{a}  \bigg{)} -a +1 =-\frac{1}{2}(b^2 +a -1)>0,
\end{split}
\end{equation}
provided that $b^2 < -a+1$. This completes the proof.
\end{proof}

\bigskip

\textbf{(i)} $\mathbf{b_0\geq 0}$ \textbf{case:} We find $a_0 \in (0, \frac{1}{2})$ and $b_0 \in [0, \frac{1}{2}]$ such that the solution of \eqref{3by3_exp} exists for all time. In the following series of lemmata, we show that there exists $(a_0 , b_0, B_0 =1)$ (under the $S_1 =0$ surface) such that the trajectory $(a(t), b(t), B(t))$ from it stays under the $S_1=0$ surface during a short time interval. In addition to this, we find a positive minimum rate of change of $b(t)$, which in turn ensures $(a(t), b(t), B(t)) \in \Omega_0$ for some $t >0$.

\begin{lemma}\label{t_star}
Let $0 < a_0 < \frac{1}{2}$ and $0\leq b_0 \leq \frac{1}{2}$. Consider the solution $(a(t) , b(t), B(t))$ of \eqref{3by3_exp} with initial data $(a_0 , b_0 , B_0 =1)$. It holds that 
$$B(t) < \frac{1}{2} \bigg{(} \frac{1}{a^2 (t)} - \frac{1}{a(t)} \bigg{)},$$
for $t \in [0, t^*)$, or $(a(t), b(t), B(t)) \in \Omega_0$, or both. Here
$$t^* := \log \bigg{[} \frac{1}{2} \bigg{(}  \frac{1}{a^2 _0} -\frac{1}{a_0}\bigg{)}  \bigg{]}.$$
\end{lemma}
\begin{proof}
The inequality holds when $t=0$, because $a_0 < \frac{1}{2}$ and $B_0 =1$.
Also, we note that $t^*>0$ because $a_0 < 1/2$. In addition to this, the second equation in \eqref{3by3_exp} implies
$$a(t) = a_0 e^{-\int^t _0 b(\tau) \, d \tau}.$$

Since $a_0 \in (0, \frac{1}{2})$ and $b_0 \in [0, \frac{1}{2}]$, by Lemma \ref{b_inc}, $b(t)$ is strictly increasing and $a(t)$ is decreasing until $(a(t), b(t), B(t))$ touches either surfaces $S_1=0$ or $S_2 =0$. 

Suppose on the contrary that $s \in (0, t^*)$ is the earliest time where 
$$B(s) =\frac{1}{2} \bigg{(} \frac{1}{a^2 (s)} - \frac{1}{a(s)}  \bigg{)},$$
and $b(s) \leq \frac{1}{2}$. This implies
$$e^s = \frac{1}{2} \bigg{\{} \bigg{(} \frac{1}{a_0} e^{\int^s _0 b(\tau) \, d \tau} \bigg{)}^2   -\frac{1}{a_0} e^{\int^s _0 b(\tau) \, d \tau} \bigg{\}}.$$
Since $b(t)$ is strictly increasing, it holds
\begin{equation}\label{a0b0_ineq}
\frac{1}{a_0} e^{\int^s _0 b(\tau) \, d \tau} > \frac{1}{a_0} e^{b_0 s} \geq \frac{1}{a_0} >2.
\end{equation}
Thus,
$$e^s > \frac{1}{2} \bigg{\{} \bigg{(} \frac{1}{a_0} e^{b_0 s} \bigg{)}^2   -\frac{1}{a_0} e^{b_0 s} \bigg{\}}. $$
On the other hand, since $s < t^* =\log \big{[} \frac{1}{2} \big{(}  \frac{1}{a^2 _0} -\frac{1}{a_0}\big{)}  \big{]} $,
$$e^s< \frac{1}{2} \bigg{(}  \frac{1}{a^2 _0} -\frac{1}{a_0}\bigg{)} \leq  \frac{1}{2} \bigg{\{} \bigg{(} \frac{1}{a_0} e^{b_0 s} \bigg{)}^2   -\frac{1}{a_0} e^{b_0 s} \bigg{\}},$$
where the inequality hold by \eqref{a0b0_ineq}. This gives the contradiction and concludes the proof.
\end{proof}

\begin{figure}[ht]
\begin{center}
\includegraphics[width=140mm]{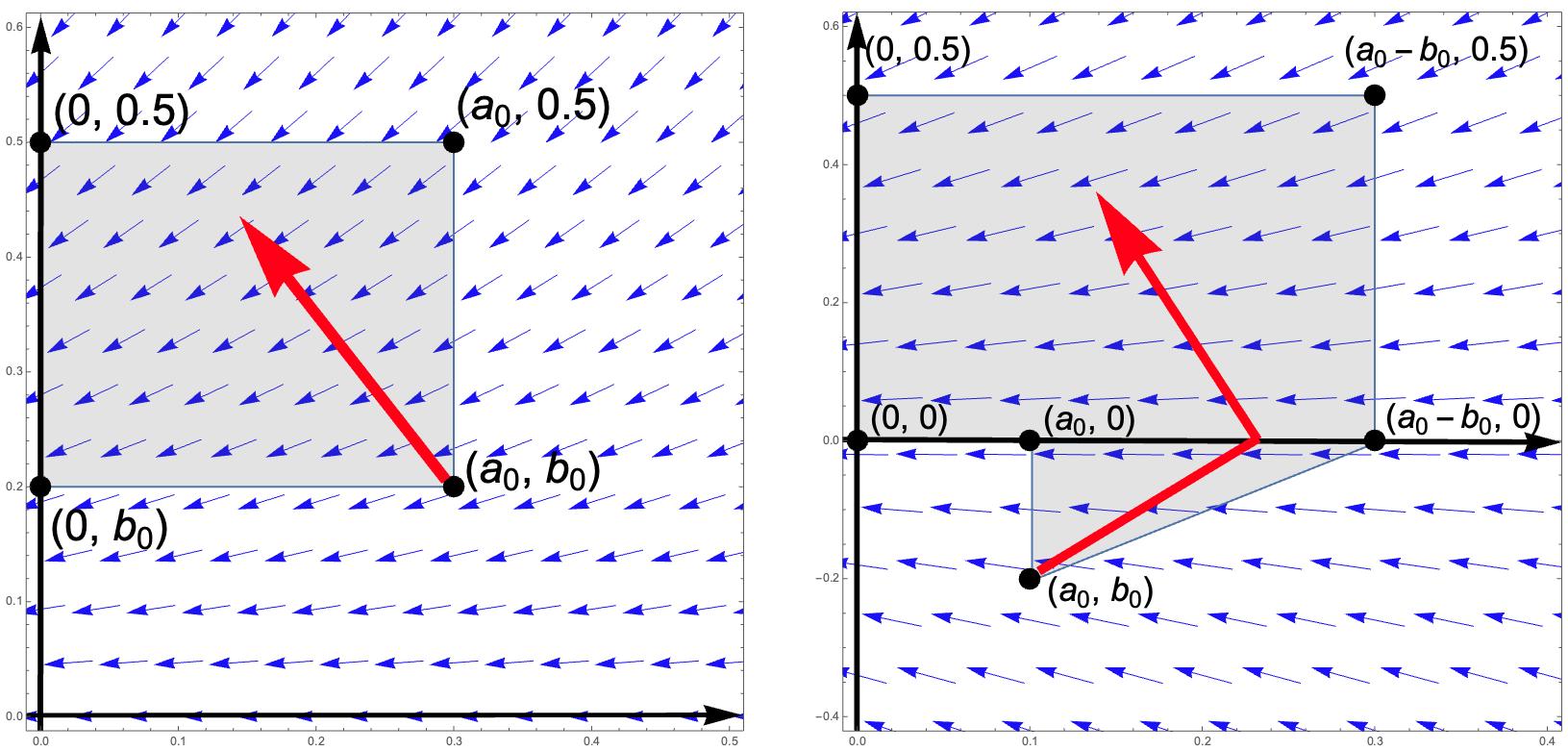}
\end{center}
\caption{Left: The rectangular region in Lemma \ref{b_dot_bound}, Right: The hexagon region in Lemmata \ref{t_starstar} and \ref{b_dot_bound_neg}, Both: Blue arrows denote $\nabla \big{(} -\frac{1}{2}b^2 -\frac{1}{2}a +\frac{1}{2} \big{)}$, and red arrows denote possible trajectories $(a(t), b(t))$ in both lemmata.}\label{gradient}
\end{figure}

\begin{lemma}\label{b_dot_bound}
Let $0 < a_0 < \frac{1}{2}$ and $0\leq b_0 \leq \frac{1}{2}$. Consider the solution $(a(t) , b(t), B(t))$ of \eqref{3by3_exp} with initial data $(a_0 , b_0 , B_0 =1)$. It holds
\begin{equation}\label{b_rate}
\dot{b}(t) \geq \frac{3}{8} - \frac{1}{2}a_0,
\end{equation}
as long as $b(t) \leq \frac{1}{2}$ and $B(t) < \frac{1}{2}\big{(} \frac{1}{a^2 (t)} - \frac{1}{a(t)}  \big{)}$.
\end{lemma}
\begin{proof}
We reuse the inequality in \eqref{b_dot_ineq},
$$\dot{b} > -\frac{1}{2}b^2 -\frac{1}{2}a +\frac{1}{2}.$$
Note that $a(t)$ is decreasing and $b(t)$ is increasing until $(a(t), b(t), B(t))$ touches either $S_1 =0$ or $S_2 =0$ surfaces. Thus, $(a(t), b(t))$ is contained in the rectangular region with vertices $(a_0 ,b _0)$, $(a_0, \frac{1}{2})$, $(0, \frac{1}{2})$, and $(0,b_0)$, as shown in Figure \ref{gradient}.
Since  $\nabla \big{(} -\frac{1}{2}b^2 -\frac{1}{2}a +\frac{1}{2} \big{)} =\langle -\frac{1}{2}, -b \rangle$, on the rectangular region, the function $ -\frac{1}{2}b^2 -\frac{1}{2}a +\frac{1}{2}$ has its minimum at $(a_0 , \frac{1}{2})$. Thus, we obtain
$$\dot{b} \geq -\frac{1}{2}\big{(}\frac{1}{2}\big{)}^2 - \frac{1}{2} a_0 + \frac{1}{2}=\frac{3}{8} - \frac{1}{2}a_0.$$
\end{proof}

Finally, we are left with finding $(a_0 , b_0, B_0 =1)$ with $0 < a_0 < \frac{1}{2}$ and $0 \leq b_0 < \frac{1}{2}$ for which $(a(t), b(t), B(t))$ enters the invariant space $\Omega_0$ \emph{before} it hits $S_1 =0$ surface, i.e., \emph{before} $t^* =\log \big{[} \frac{1}{2} \big{(}  \frac{1}{a^2 _0} -\frac{1}{a_0}\big{)}  \big{]} $ in Lemma \ref{t_star}.

Integrating \eqref{b_rate} gives
$$b(t) \geq \bigg{(}  \frac{3}{8} -\frac{1}{2}a_0 \bigg{)}t + b_0.$$
This inequality leads to $b(t)=\frac{1}{2}$, no later than
$$s:= \frac{(\frac{1}{2} -b_0)}{(\frac{3}{8} -\frac{1}{2}a_0)},$$
unless $(a(t), b(t), B(t))$ touches $S_1 =0$ surface. Therefore, we require $s \leq t^*$, or
\begin{equation}\label{omega_m}
\frac{(\frac{1}{2} -b_0)}{(\frac{3}{8} -\frac{1}{2}a_0)} \leq \log \bigg{[} \frac{1}{2} \bigg{(}  \frac{1}{a^2 _0} -\frac{1}{a_0}\bigg{)}  \bigg{]}.
\end{equation}

\bigskip
\textbf{(ii)} $\mathbf{b_0< 0}$ \textbf{case:} Next, we find $a_0 \in (0, \frac{1}{2})$ and $b_0 <0$ such that the solution of \eqref{3by3_exp} exists for all time. The idea of proof is similar to that of $b\geq 0$ case.

\begin{lemma}\label{t_starstar}
Let $0 < a_0 < \frac{1}{2}$ and $a_0 -\frac{1}{2} < b_0 < 0$. Consider the solution $(a(t) , b(t), B(t))$ of \eqref{3by3_exp} with initial data $(a_0 , b_0 , B_0 =1)$. It holds that 
$$B(t) < \frac{1}{2} \bigg{(} \frac{1}{a^2 (t)} - \frac{1}{a(t)} \bigg{)},$$
for $t \in [0, t^{**})$, or $(a(t), b(t), B(t)) \in \Omega$, or both. Here
$$t^{**} := \log \bigg{[} \frac{1}{2} \bigg{(}  \frac{1}{(a_0 -b_0)^2} -\frac{1}{a_0 -b_0}\bigg{)}  \bigg{]}.$$
\end{lemma}
\begin{proof}
The inequality holds when $t=0$, because $a_0 < \frac{1}{2}$ and $B_0 =1$. Further, $t^{**}>0$ because $a_0 - b_0 <\frac{1}{2}$. Since $0 < a_0 < \frac{1}{2}$ and $a_0 -\frac{1}{2} <b_0 < 0$, by Lemma \ref{b_inc}, $b(t)$ is strictly increasing until $(a(t), b(t), B(t))$ touches either surfaces $S_1 =0$ or $S_2 =0$. Also, $a(t)$ is strictly increasing if $b(t)<0$; and $a(t)$ is strictly decreasing if $b(t)>0$. 

In particular, when $b(t)<0$ and $B(t)<\frac{1}{2}\big{(}  \frac{1}{a^2 (t)} - \frac{1}{a(t)} \big{)}$, it holds
$$\dot{b} > \dot{a}$$
in the triangular region with vertices $(0,0)$, $(0, -\frac{1}{2})$ and $(\frac{1}{2}, 0)$. Indeed, consider
\begin{equation*}
\begin{split}
\dot{b} - \dot{a} &=-\frac{1}{2}b^2 - B a^2 -a +1 +ba\\
&> -\frac{1}{2}b^2 - \frac{a^2}{2}\bigg{(} \frac{1}{a^2} - \frac{1}{a} \bigg{)} -a + 1 +ba\\
&=-\frac{1}{2}b^2 -\frac{a}{2} + ba + \frac{1}{2}\\
&=-\frac{1}{2}\big{(} b -a + \sqrt{1-a+a^2} \big{)}\big{(} b -a - \sqrt{1-a+a^2} \big{)} >0,
\end{split}
\end{equation*}
when $a - \sqrt{1-a+a^2} < b \leq 0$. This implies that $(a(t), b(t))$ stays in the irregular hexagon region with vertices $(a_0 , b_0)$, $(a_0 - b_0 , 0)$, $(a_0 - b_0, \frac{1}{2})$, $(0, \frac{1}{2})$, $(0,0)$, and $(a_0, 0)$ until $(a(t), b(t), B(t))$ touches $S_1 =0$ surface, if any. See Figure \ref{gradient}. In particular, 
$$a(t) < a_0 - b_0.$$

Now, let's suppose that $s \in (0, t^{**})$ is the earliest time where $B(s) = \frac{1}{2} \big{(}  \frac{1}{a^2 (s)}  - \frac{1}{a(s)} \big{)}$. This means
$$e^s =  \frac{1}{2} \bigg{(}  \frac{1}{a^2 (s)}  - \frac{1}{a(s)} \bigg{)}.$$
In addition to this, since $\frac{1}{2} \big{(} \frac{1}{a^2} - \frac{1}{a} \big{)}$ is decreasing in $a$, $a(s) < a_0 - b_0$ implies
$$\frac{1}{2} \bigg{(}  \frac{1}{a^2 (s)}  - \frac{1}{a(s)} \bigg{)} > \frac{1}{2} \bigg{(}  \frac{1}{(a_0 - b_0)^2}  - \frac{1}{a_0 - b_0} \bigg{)}. $$
Thus, 
$$e^s > \frac{1}{2} \bigg{(}  \frac{1}{(a_0 - b_0)^2}  - \frac{1}{a_0 - b_0} \bigg{)} ,$$
or
$$s> \log \bigg{[} \frac{1}{2} \bigg{(}  \frac{1}{(a_0 -b_0)^2} -\frac{1}{a_0 -b_0}\bigg{)}  \bigg{]} = t^{**},$$
this contradicts to $s<t^{**}$ and concludes the proof.
\end{proof}

\begin{lemma}\label{b_dot_bound_neg}
Let $0 < a_0 < \frac{1}{2}$ and $a_0 -\frac{1}{2} <b_0 <0$. Consider the solution $(a(t) , b(t), B(t))$ of \eqref{3by3_exp} with initial data $(a_0 , b_0 , B_0 =1)$. It holds
\begin{equation}\label{b_rate_bneg}
\dot{b}(t) \geq \frac{3}{8} - \frac{1}{2}(a_0 - b_0),
\end{equation}
as long as $b(t) \leq \frac{1}{2}$ and $B(t) < \frac{1}{2}\big{(} \frac{1}{a^2 (t)} - \frac{1}{a(t)}  \big{)}$.
\end{lemma}
\begin{proof}
We reuse the argument in the proof Lemma \ref{b_dot_bound}. From the inequality in \eqref{b_dot_ineq},
$$\dot{b} > -\frac{1}{2}b^2 -\frac{1}{2}a +\frac{1}{2}.$$
Note that $a(t)$ is increasing if $b(t)<0$, and decreasing if $b(t)>0$; also $b(t)$ is increasing until $(a(t), b(t), B(t))$ touches either $S_1 =0$ or $S_2 =0$ surfaces. In addition to this, $\dot{b} > \dot{a}$ when $b(t) <0$.
 Thus, as shown in Figure \ref{gradient}, $(a(t), b(t))$ is contained in the irregular hexagon region with vertices $(a_0 , b_0)$, $(a_0 - b_0 , 0)$, $(a_0 - b_0, \frac{1}{2})$, $(0, \frac{1}{2})$, $(0,0)$, and $(a_0, 0)$ until $(a(t), b(t), B(t))$ touches $S_1 =0$ surface, if any. Since  $\nabla \big{(} -\frac{1}{2}b^2 -\frac{1}{2}a +\frac{1}{2} \big{)} =\langle -\frac{1}{2}, -b \rangle$, on the hexagon region, the function $ -\frac{1}{2}b^2 -\frac{1}{2}a +\frac{1}{2}$ has its minimum at $(a_0 -b_0 , \frac{1}{2})$. Thus, we obtain
$$\dot{b} \geq -\frac{1}{2}\big{(}\frac{1}{2}\big{)}^2 - \frac{1}{2}(a_0 -b_0 )+ \frac{1}{2}=\frac{3}{8} - \frac{1}{2}(a_0 -b_0).$$
\end{proof}

Finally, we are left with finding $(a_0 , b_0, B_0 =1)$ with $0 < a_0 < \frac{1}{2}$ and $a_0 -\frac{1}{2} \leq b_0 < 0$ for which $(a(t), b(t), B(t))$ enters the invariant space $\Omega_0$ \emph{before} it hits $S_1 =0$ surface, i.e., \emph{before} $t^{**} =\log \big{[} \frac{1}{2} \big{(}  \frac{1}{(a_0 - b_0)^2} -\frac{1}{a_0 - b_0}\big{)}  \big{]} $ in Lemma \ref{t_starstar}.

Integrating \eqref{b_rate_bneg} gives
$$b(t) \geq \bigg{\{}  \frac{3}{8} -\frac{1}{2}(a_0 - b_0) \bigg{\}}t + b_0.$$
This inequality leads to $b(t)=\frac{1}{2}$, no later than
$$s:= \frac{(\frac{1}{2} -b_0)}{\{ \frac{3}{8} -\frac{1}{2}(a_0 - b_0)  \}},$$
unless $(a(t), b(t), B(t))$ touches $S_1 =0$ surface. Therefore, we require $s \leq t^{**}$, or
\begin{equation}\label{omega_b}
\frac{(\frac{1}{2} -b_0)}{\{ \frac{3}{8} -\frac{1}{2}(a_0 - b_0)  \}}\leq \log \bigg{[} \frac{1}{2} \bigg{(}  \frac{1}{(a _0 - b_0)^2} -\frac{1}{a_0- b_0}\bigg{)}  \bigg{]}.
\end{equation}

\bigskip 
\textbf{(iii)} \textbf{Union of} $\mathbf{b_0\geq 0}$ \textbf{and} $\mathbf{b_0< 0}$ \textbf{cases:}
In order to find a set of sub-critical data for \eqref{ode_exp}, we collect the sets described in \eqref{omega_0} and \eqref{omega_m}, as well as \eqref{omega_b}.
The three condition mentioned above lead to
$$\Omega^T : = \Omega_0 \bigg{|}_{B=1} := \{(a,b) \ | \ 0 <a \leq 1/2, \ b \geq 1/2 \},$$
$$\Omega^{M} : = \bigg{\{}(a,b) \ \bigg{|} \ 0 <a < 1/2, \ \max \bigg{\{} 0, \frac{1}{2} - \bigg{(}\frac{3}{8}  - \frac{a}{2}\bigg{)}\log\bigg{[} \frac{1}{2} \big{(} \frac{1}{a^2}  - \frac{1}{a} \big{)} \bigg{]} \bigg{\}} \leq b \leq \frac{1}{2} \bigg{\}},$$
and
$$\Omega^{B} := \bigg{\{ } (a,b) \ \bigg{|} \ 0 <a <1/2, \ a -\frac{1}{2} < b <0,   \ \frac{(\frac{1}{2} -b)}{\{ \frac{3}{8} -\frac{1}{2}(a - b)  \}}\leq \log \bigg{[} \frac{1}{2} \bigg{(}  \frac{1}{(a  - b)^2} -\frac{1}{a- b}\bigg{)}  \bigg{]} \bigg{\}},$$
respectively, as shown in Figure \ref{ot_om_ob}. Note that $\Omega^T$, $\Omega^{M}$ and $\Omega^{B}$ described above are slightly different (up to boundaries) from the ones in Theorem \ref{thm_main}. This is because the open set $\Omega^{TMB} \backslash \partial \Omega^{TMB}$ is needed later when we perform comparisons between \eqref{ode_exp} and the original system \eqref{ode2_intro}.

We claim that $(a_0, b_0)\in \Omega^{TMB}:= \Omega^T \cup \Omega^M \cup \Omega^B$ admits a global solution to \eqref{ode_exp}.


\begin{lemma}\label{lemma_invariant}
$(a_0, b_0) \in \Omega^{TMB}$ leads to
$$0 < a(t) \leq 1/2 , \ and \ -\frac{1}{2}  \leq b(t) \leq\max \{ |b_0| , \sqrt{2} \} , \ \ \forall t \in [0, \infty).$$ 
\end{lemma}
\begin{proof}
Since we showed that $(a(t), b(t)) \in \Omega^T$ in finite time, it suffices to show that $b(t) \leq \max \{ |b_0 |, \sqrt{2} \},  \forall t \in [0, \infty)$. Consider
$$\dot{b} = -\frac{1}{2} b^2 -e^t a^2 -a +1 \leq -\frac{1}{2}b^2 -a^2 -a +1 \leq -\frac{1}{2}b^2 + \max_{0 \leq a \leq 1/2} \big{\{} -a^2 -a +1 \big{\}}.$$
Thus, $\dot{b} \leq -\frac{1}{2}b^2 + 1$ gives the desired result.
\end{proof}


Now, the last step of the proof is to compare
\begin{equation}\label{comp_epsystem}
\left\{
  \begin{array}{ll}
    \dot{d} = -\frac{1}{2}d^2 +A(t) \rho^2 -\rho +1, \\
    \dot{\rho} = -d\rho,\\
  \end{array}
\right.
\end{equation}
with
\begin{equation}\label{comp_auxsystem}
\left\{
  \begin{array}{ll}
    \dot{b} = -\frac{1}{2}b^2  -e^t a^2 -a +1, \\
    \dot{a} = -ba.\\
  \end{array}
\right.
\end{equation}
We recall that
$$-e^t \leq A(t) \leq  \frac{1}{2}\bigg{(} \frac{\omega_0}{\rho_0} \bigg{)}^2, \ \ t\geq 0.$$
We show the monotonicity relation between two ODE systems:
\begin{lemma}\label{lemma_comp}
\begin{equation*}
\left\{
  \begin{array}{ll}
	b(0) < d(0), \\
    0<\rho(0)<a(0)\\
  \end{array}
\right.
implies
\ 
\left\{
  \begin{array}{ll}
	b(t) < d(t), \\
    0<\rho(t)<a(t)\\
  \end{array}
\right.
\ for \ all \ t>0.
\end{equation*}
\end{lemma}

\begin{proof}
Suppose $t_1$ is the earliest time when the above assertion is violated. Consider
$$a(t_1) = a(0) e^{-\int^{t_1} _ 0 b(\tau) \, d \tau} > \rho(0) e^{-\int^{t_1} _0 d(\tau) \, d \tau} = \rho (t_1).$$
Therefore, it is left with only one possibility that $d(t_1) = b(t_1).$ Consider
\begin{equation}\label{diff_systems}
\dot{b}-\dot{d} =-\frac{1}{2} (b^2 -d^2) - e^t a^2 - A(t) \rho^2 - a + \rho.
\end{equation}
Since $b(t)-d(t) <0$ for $t<t_1$ and $b(t_1) - d(t_1)=0$, hence at $t=t_1$, we have
$$\dot{b}(t_1)-\dot{d} (t_1) \geq 0.$$
But the right hand side of \eqref{diff_systems},  when it is evaluated at $t=t_1$, is negative. Indeed
\begin{equation*}
\begin{split}
&-\frac{1}{2} (b^2 (t_1) -d^2 (t_1)) - e^{t_1} a^2 (t_1) - A(t_1) \rho^2 (t_1) - a(t_1) + \rho(t_1)\\
&= - e^{t_1} a^2 (t_1) - A(t_1) \rho^2 (t_1) - a(t_1) + \rho(t_1)\\
&=e^{t_1} \big{(}  -a^2 (t_1) + \rho^2 (t_1) \big{)} +\rho^2 (t_1) \big{(} -e^{t_1} -A(t_1)  \big{)} - a(t_1) + \rho(t_1).
\end{split}
\end{equation*}
Thus $a(t_1) > \rho(t_1)$ and $-e^{t_1} \leq A(t_1)$ give the desired result. This leads to the contradiction.
\end{proof}

\begin{lemma}\label{lemma_dbound}
 Consider \eqref{comp_epsystem}. If there exists $\rho_M  >0$ such that $\rho(t) \leq \rho_M$, $\forall t \geq 0$, then $d(t)$ is bounded from above for all $d_0 \in \mathbb{R}$.
\end{lemma}
\begin{proof}
Since $A(t)\leq   \frac{1}{2}( \frac{\omega_0}{\rho_0} )^2 =:\gamma$, we have
\begin{equation*}
\begin{split}
\dot{d}&= -\frac{1}{2}d^2 + A(t)\rho^2 - \rho +1\\
&\leq  -\frac{1}{2}d^2 +    \gamma \rho^2 - \rho +1\\
&\leq  -\frac{1}{2}d^2  + \max\{ 1, \gamma\rho^2 _M - \rho_M +1 \}.
\end{split}
\end{equation*}
Thus,
$$d(t) \leq \max \big{\{} d_0, \sqrt{2\max\{ 1, \gamma \rho^2 _M - \rho_M +1 \} }  \big{\}}.$$
\end{proof}

We finally combine the comparison principle in Lemma \ref{lemma_comp} with Lemma \ref{lemma_invariant}. Let
$$\Omega:= \Omega^{TMB} \backslash \partial \Omega^{TMB}.$$
Note that $\Omega$ is  an open set and given any initial data $(\rho_0 , d_0) \in \Omega$ for system \ref{comp_epsystem}, we can find $\epsilon>0$ and initial data $(a_0 , b_0):=(\rho_0 +\epsilon , d_0 -\epsilon) \in \Omega$ for system \ref{comp_auxsystem}. Therefore, by lemmata  \ref{lemma_comp} and \ref{lemma_invariant},
$$0<\rho(t) < \frac{1}{2} \ \ and \ \ - \frac{1}{2}< d(t), \ \ \forall t\geq 0.$$
In addition to this, by Lemma \ref{lemma_dbound}, $\rho(t)< \frac{1}{2}$ implies that $d(t)$ is bounded from above for all $t\geq 0$. This completes the proof of the theorem.


\bigskip
\section{Numerical Examples}\label{section 5}
In this section, we present examples and numerical simulations to illustrate that there are some initial data that lead to the global smooth solutions of two-dimensional Euler-Poisson system in \eqref{101}. We also observe the time evolution of $\Delta^{-1} (\rho -c)$ (and its derivatives) which is closely related to the decay condition  $A(t) \geq -e^{t}$ in  \eqref{exp_condi_2}.


\begin{figure}[ht!]\begin{center}
\includegraphics[width=130mm]{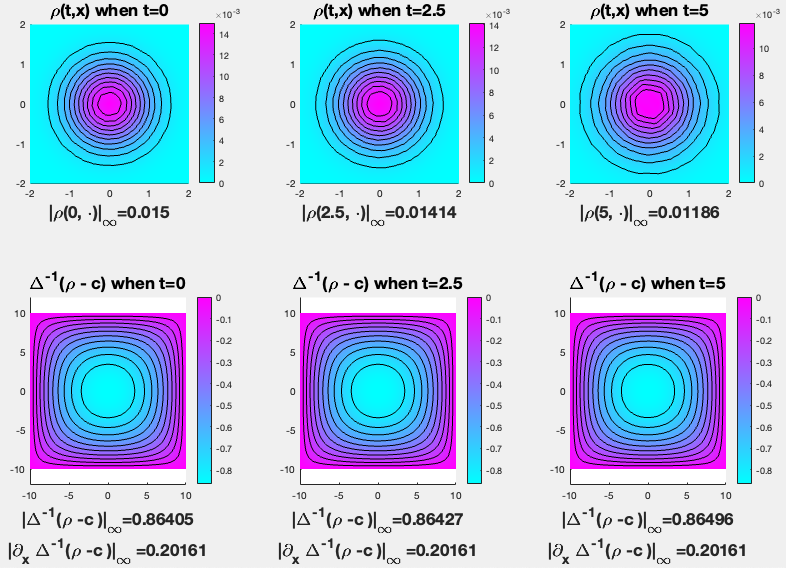}
\caption{Example \ref{eg5.1}: Time evolutions (left to right) of $\rho(t,x)$ and $\Delta^{-1}(\rho -c_b)$.}\label{eg1}
\end{center}
\end{figure}

\begin{example}\label{eg5.1}
(attractive, non-zero background) We consider \eqref{101} with $k=-1$, $c_b=0.03$ subject to the following initial data: $\mathbf{u}(0,\mathbf{x})=\mathbf{0}$ and
$$\rho(0, \mathbf{x}) = \frac{0.03}{2}\exp(-|\mathbf{x}|^2).$$
It is observed that $|\rho(t, \cdot)|_{\infty}$, $|\Delta^{-1}(\rho -c)|_{\infty}$ are strictly decreasing in time. It is interesting to notice that $|\frac{\partial}{\partial x} \Delta^{-1}(\rho -c)   |_{\infty}$ stays the same in time. From the graphs in Figure \ref{eg1}, we anticipate that the time growth of $|R_{ij}(\rho -c)|_{\infty}$ and $|A(t)|$ should be mild. 

We should point out that the initial vorticity $\omega =\nabla \times \mathbf{u}(0,\cdot)$ vanishes in this example, which in turn $\omega(t) \equiv 0$ for all time due to \eqref{Veqn}.  Thus the divergence equation in \eqref{Deqn} is reduced to
$$d'=-\frac{1}{2}d^2 -\frac{1}{2}\eta^2  -\frac{1}{2}\xi^2 + k(\rho-c_b).$$
The simulation results in Examples \ref{eg5.1} and \ref{eg5.2} are interesting: In addition to $-\frac{1}{2}d^2$, \emph{all terms except} $-kc_b$ intensify the blow-up behavior of $d$. Thus, with vanishing initial vorticity condition, it is natural to anticipate blow-up for a large class of initial data, e.g. \cite{CT08}. However, the numerical simulations in two examples demonstrate global existences.

\end{example}

\begin{figure}[ht!]\begin{center}
\includegraphics[width=130mm]{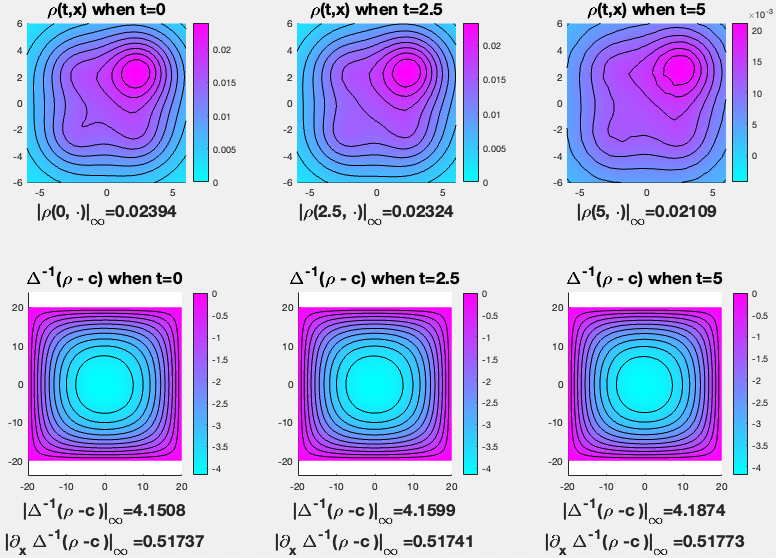}
\caption{Example \ref{eg5.2}: Time evolutions (left to right) of $\rho(t,x)$ and $\Delta^{-1}(\rho -c_b)$.}\label{eg2}
\end{center}
\end{figure}

\begin{example}\label{eg5.2}
(attractive, non-zero background) We consider \eqref{101} with $k=-1$, $c_b=0.04$ subject to the following initial data: $\mathbf{u}(0,\mathbf{x})=\mathbf{0}$ and non symmetric density
\begin{equation*}
\begin{split}
\rho(0,\mathbf{x}) = 0.01\sech &\big{(}0.5 |\mathbf{x} + \langle 2.5, 2.5 \rangle| \big{)} + 0.02\sech \big{(}0.5 |\mathbf{x} - \langle 2.5, 2.5 \rangle|\big{)}\\
&+ 0.01\sech \big{(}0.5 |\mathbf{x} - \langle -2.5, 2.5 \rangle| \big{)}+0.01\sech \big{(}0.5 |\mathbf{x} - \langle 2.5, -2.5 \rangle| \big{)}.
\end{split}
\end{equation*}
It is observed that $|\rho(t, \cdot)|_{\infty}$ is strictly decreasing in time. $|\Delta^{-1}(\rho -c)|_{\infty}$ and $|\frac{\partial}{\partial x} \Delta^{-1}(\rho -c)   |_{\infty}$ are increasing in time, but their growth rate is very mild. See Figure \ref{eg2}.
\end{example}

\begin{figure}[ht!]\begin{center}
\includegraphics[width=130mm]{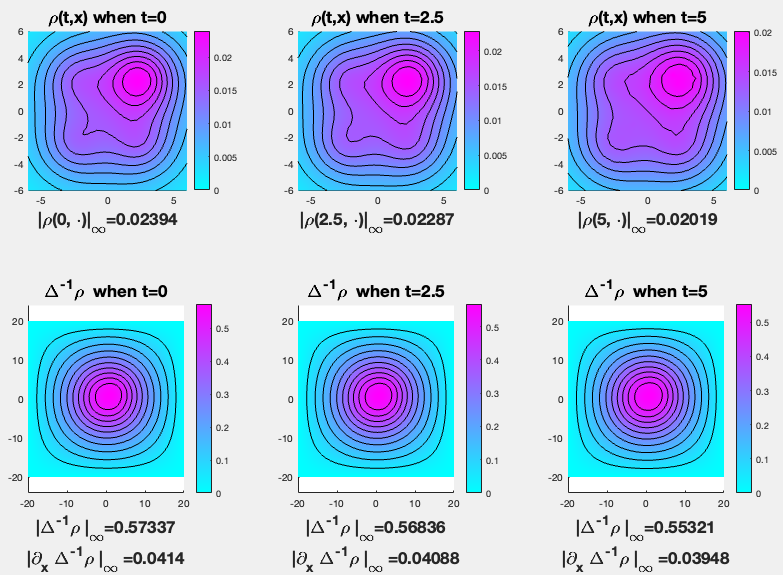}
\caption{Example \ref{eg5.3}: Time evolutions (left to right) of $\rho(t,x)$ and $\Delta^{-1}\rho $.}\label{eg3}
\end{center}
\end{figure}

\begin{example}\label{eg5.3}
(repulsive, zero background) We consider \eqref{101} with $k=1$, $c_b=0$ subject to the initial data in Example \ref{eg5.2}.
It is observed that $|\rho(t, \cdot)|_{\infty}$, $|\Delta^{-1}\rho |_{\infty}$ and $|\frac{\partial}{\partial x} \Delta^{-1}\rho    |_{\infty}$ are strictly decreasing in time. See Figure \ref{eg3}.
\end{example}


\clearpage
\bibliographystyle{abbrv}

\end{document}